\documentclass[12pt,reqno]{amsart}

\usepackage{amssymb}
\usepackage{amscd}
\usepackage{amsfonts}
\usepackage{setspace}
\usepackage{version}

\usepackage{graphicx}

\newtheorem{theorem}{Theorem}[section]

\newtheorem{proposition}[theorem]{Proposition}

\theoremstyle{definition}

\renewcommand{\leq}{\leqslant}
\renewcommand{\geq}{\geqslant}

\def\R{\mathbb{R}}

\def\Z{\mathbb{Z}}

\def\overast{\overline{\ast}}

\newcommand{\md}[1]{\ensuremath{(\operatorname{mod}\, #1)}}

\numberwithin{equation}{section}

\begin{document}

\title[Weighted Pr\'ekopa--Leindler and quasicubes]{A Weighted Pr\'ekopa--Leindler inequality and sumsets with quasicubes}



\author[Green]{Ben Green}

\author[Matolcsi]{D\'avid Matolcsi}

\author[Ruzsa]{Imre Ruzsa}

\author[Shakan]{George Shakan}

\author[Zhelezov]{Dmitrii Zhelezov}

\thanks{BG and GS are supported by BG's Simons Investigator Grant number 376201.}

\subjclass[2000]{Primary }

\maketitle
\begin{abstract}
We give a short, self-contained proof of two key results from a paper of four of the authors. The first is a kind of weighted discrete Pr\'ekopa-Leindler inequality. This is then applied to show that if $A, B \subseteq \Z^d$ are finite sets and $U$ is a subset of a ``quasicube'' then $|A + B + U| \geq |A|^{1/2} |B|^{1/2} |U|$. This result is a key ingredient in forthcoming work of the fifth author and P\"alv\"olgyi on the sum-product phenomenon.
\end{abstract}

\section{Introduction}

\emph{Quasicubes.} The notion of a quasicube $\Sigma \subseteq \Z^d$ is defined inductively. When $d = 1$, a quasicube is simply a set of size two. For larger $d$, $\Sigma$ is a quasicube if 
\begin{enumerate} \item $\pi (\Sigma) = \{x_0, x_1\}$ is a set of size two, where $\pi : \Z^d \rightarrow \Z$ is the coordinate projection onto the final coordinate, and
\item The fibre $\Sigma_i := \Sigma \cap \pi^{-1}(x_i)$ (considered as a subset of $\Z^{d-1}$) is a quasicube.
\end{enumerate}

Thus, for instance, the usual cube $\{0,1\}^d$ is a quasicube. Another example of a quasicube with $d = 2$ is the set $\Sigma = \{(0,0), (1,0), (0,1), (1,2)\}$.

The following result is established in \cite{mrsz}. 

\begin{theorem}\label{mainthm}
Let $A, B \subseteq \Z^d$ be finite sets and suppose that $U \subseteq \Z^d$ is contained in a quasicube. Then $|A + B + U| \geq |A|^{1/2} |B|^{1/2} |U|$.
\end{theorem}

Our aim in this note is to give a short, self-contained proof of this result.

\section{A weighted discrete Pr\'ekopa--Leindler inequality}

As in \cite{mrsz}, we deduce Theorem \ref{mainthm} from a weighted discrete Pr\'ekopa--Leindler inequality. Let $a, b : \Z \rightarrow [0,\infty)$ be compactly supported functions. We define the \emph{max-convolution}
\[ a \overast b(n) := \sup_{m \in \Z} a(n - m) b(m),\] and we write
\[ \Vert a \Vert_2 := \big( \sum_n a(n)^2 \big)^{1/2}, \quad \Vert b \Vert_2 := \big( \sum_n b(n)^2 \big)^{1/2}.\]

The following result is equivalent to \cite[Theorem 11.1]{mrsz}.

\begin{proposition}\label{prop01}
Let $a, b : \Z \rightarrow [0,\infty)$ be compactly supported functions and let $p \in [0,1]$. Then we have
\[ \sum_n \max(p a \overast b(n), (1 - p) a \overast b (n-1)) \geq \Vert a \Vert_2 \Vert b \Vert_2.\]
\end{proposition}

In the case $p = \frac{1}{2}$, this is (2.4) in the paper of Pr\'ekopa \cite{pre}, where it is used to establish the 1-dimensional case of what is now known as the Pr\'ekopa--Leindler inequality (we will recall the statement of this below). We will proceed in the opposite direction, deducing Proposition \ref{prop01} from Pr\'ekopa--Leindler.

Suppose that $f, g : \R \rightarrow [0, \infty)$ are compactly supported, piecewise continuous functions. Then the (1-dimensional) Pr\'ekopa--Leindler inequality states that
\begin{equation}\label{pl}  \int f \overast g \geq 2 \Vert f \Vert_2 \Vert g \Vert_2, \end{equation} where the max-convolution is defined by
\[ f \overast g(x) := \sup_{y \in \R} f(x - y)g(y),\] and the norms are the usual Lebesgue norms
\[ \Vert f \Vert_2 := \big( \int f^2 \big)^{1/2} , \quad \Vert g \Vert_2 := \big( \int g^2 \big)^{1/2}.\]
(It should always be clear from context whether we are applying $\overast$ or $\Vert \cdot \Vert_2$ with functions on $\Z$ or functions on $\R$). We note that Brascamp and Lieb \cite{brascamp-lieb} found a much shorter proof of \eqref{pl} than the original (see also this survey of Gardner \cite{Ga}).

\begin{proof}[Proof of Proposition \ref{prop01}] By continuity we may assume that $p \in (0,1)$. Set $\lambda := \log(\frac{1}{p} - 1)$. 
Apply \eqref{pl} with functions $f, g$ defined by
\[ f(x) := e^{\lambda \{x \}} a (\lfloor x\rfloor), \quad g(y) := e^{\lambda \{y\}} b(\lfloor y \rfloor).\]

Let $n \in \Z$ and $0 \leq t < 1$. Suppose that $x + y = n + t$. Then, since $x - 1 < \lfloor x\rfloor \leq x$, we have $n-2 < \lfloor x\rfloor + \lfloor y \rfloor < n+1$, or in other words $\lfloor x\rfloor + \lfloor y \rfloor = n-1$ or $n$. If $\lfloor x \rfloor + \lfloor y \rfloor = n - 1$ then 
\[ f(x) g(y) \leq e^{\lambda(t + 1)} a \overast b (n-1),\] whilst if $\lfloor x \rfloor + \lfloor y \rfloor = n$ then 
\[ f(x) g(y) \leq e^{\lambda t} a \overast b (n).\]
Therefore
\[ f \overast g (n + t) \leq e^{\lambda t} \max(a \overline{\ast} b (n), e^{\lambda} a \overline{\ast} b(n - 1)).\]
Integrating over $t \in [0,1)$ and then summing over $n \in \Z$ yields
\begin{equation}\label{star-1} \int f \overast g \leq \frac{e^{\lambda} - 1}{\lambda} \sum_n \max(a \overline{\ast} b (n), e^{\lambda} a \overline{\ast} b(n - 1)).\end{equation}
On the other hand,
\[ \Vert f \Vert_2^2 = \frac{e^{2\lambda} - 1}{2\lambda} \Vert a \Vert_2^2, \quad \Vert g \Vert_2^2 = \frac{e^{2\lambda} - 1}{2\lambda} \Vert b \Vert_2^2.\] Substituting into \eqref{pl} gives
\[ \sum_n \max(a \overline{\ast} b (n), e^{\lambda} a \overline{\ast} b(n - 1)) \geq (e^{\lambda} + 1) \Vert a \Vert_2 \Vert b \Vert_2.\] Recalling the choice of $\lambda$ (thus $p = \frac{1}{e^{\lambda} + 1}$), the proposition follows.\end{proof}

\section{Proof of the main theorem}

The arguments of this section are all in \cite{mrsz}, but there they form part of a more general framework. Here we provide a self-contained account tailored to the specific purpose of proving Theorem \ref{mainthm}.

\begin{proof}[Proof of Theorem \ref{mainthm}]

We proceed by induction on $d$. The proof of the inductive step also proves the base case $d = 1$.

Suppose that $U$ is contained in a quasicube $\Sigma \subset \Z^d$. Suppose that $\pi(\Sigma) = \{x_0, x_1\}$, where $\pi : \Z^d \rightarrow \Z$ is projection onto the last coordinate. Since the inequality is translation-invariant, we may assume that $x_0 = 0$ and $x_1 = q > 0$. Suppose first that $q = 1$.

Let $A_i := A \cap \pi^{-1}(n)$ be the fibre of $A$ above $n$, and similarly for $B$. The set $U$ has just two fibres $U_0, U_1$ and, by the definition of quasicubes, they are both contained in quasicubes of dimension $d-1$.

Observe that the fibre of $A + B + U$ above $n$ contains $A_x + B_y + U_0$ whenever $x + y = n$, and $A_x + B_y + U_1$ whenever $x + y = n-1$. By induction,
\[ |A_x + B_y + U_0| \geq |A_x|^{1/2}|B_y|^{1/2}|U_0|,\]
\[ |A_x + B_y + U_1| \geq |A_x|^{1/2}|B_y|^{1/2}|U_1|,\] and so the fibre $(A + B + U)_n$ of $A + B + U$ above $n$ has size at least
\[ \max\big(|U_0|\max_{x + y = n} |A_x|^{1/2}|B_y|^{1/2} , |U_1| \max_{x + y = n-1} |A_x|^{1/2} |B_y|^{1/2}    \big).\]
This is equal to 
\[ |U| \max \big( p a \overast b(n) + (1 - p) a \overast b(n-1)\big),\]
where $p := |U_0|/|U|$, $a(x) := |A_x|^{1/2}$ and $b(y) := |B_y|^{1/2}$.
Summing over $n$ and applying Proposition \ref{prop01} we obtain
\begin{align*}
|A + B + U| & = \sum_n |(A + B + U)_n| \\ & \geq |U| \sum_n \max \big( p a \overast b(n) + (1 - p) a \overast b(n-1)\big) \\ & \geq |U| \Vert a \Vert_2 \Vert b \Vert_2  = |U| |A|^{1/2} |B|^{1/2}. 
\end{align*}
This proves the result when $q = 1$. Suppose now that $q$ is arbitrary, and foliate $A = \bigcup_{r \in \Z/q\Z} A_r$, $B = \bigcup_{s \in \Z/q\Z} B_s$, where $A_r := \{ a \in A: \pi(a) \equiv r \md{q}\}$ and similarly for $B_s$. Let $r_*$ be such that $|A_r| \leq |A_{r_*}|$ for all $r$, and $s_*$ be such that $|B_s| \leq |B_{s_*}|$ for all $s$. 

The sets $A_{r_*} + B_s + U$ are disjoint as $s$ varies, and so by the case $q = 1$ (rescaled) we have
\begin{equation}\label{eq1} |A + B + U| \geq \sum_s |A_{r_*} + B_s + U| \geq |U| |A_{r_*}|^{1/2} \sum_s |B_s|^{1/2}.\end{equation}
Similarly, 
\begin{equation}\label{eq2} |A + B + U| \geq |U| |B_{s_*}|^{1/2}\sum_r |A_r|^{1/2}.\end{equation}
Taking products of \eqref{eq1}, \eqref{eq2} and using 
\[ |A_{r_*}|^{1/2} \sum_r |A_r|^{1/2} \geq \sum_r |A_r| = |A|,\]
\[ |B_{s_*}|^{1/2} \sum_s |B_s|^{1/2} \geq \sum_s |B_s| = |B|,\] the result follows.
\end{proof}

\end{document}